\documentclass[12pt,a4paper]{amsart}

\usepackage[latin1]{inputenc}
\usepackage[T1]{fontenc}

\usepackage[english]{babel}

\usepackage[right=1.7cm,left=1.7cm,bottom=3cm]{geometry}

\usepackage{xcolor}
\definecolor{bleufonce}{rgb}{0,0,0.4}

\usepackage{hyperref}
\hypersetup{
colorlinks=true,
linkcolor=bleufonce,urlcolor=bleufonce,
citecolor=bleufonce,filecolor=bleufonce}

\usepackage{amsmath, amssymb}
\usepackage{mathrsfs}
\usepackage{yhmath} 

\usepackage{pst-all}
\usepackage{auto-pst-pdf}

\usepackage{graphicx}

\newtheorem{theorem}{Theorem\!\!}

\newtheorem{prop}{Proposition}
\newtheorem{rem}{Remark}
\newtheorem{lemma}{Lemma\!\!}

\newtheorem{conjecture}{Conjecture\!\!}

\newcommand{\R}{\mathbb{R}}

\newcommand{\C}{\mathbb{C}}
\newcommand{\E}[1]{\mathbf E\left(#1\right)}
\newcommand{\la}{\lambda}
\newcommand{\al}{\alpha}
\newcommand{\e}{\varepsilon}

\newcommand{\et}{\text{\:}\text{ \normalfont and }\text{\:}}

\newcommand{\elaw}{\overset{\mathrm d}{=}}
\newcommand{\Ga}{\Gamma}
\newcommand{\ga}{\gamma}
\newcommand{\Exp}{\operatorname{Exp}}
\newcommand{\im}{\operatorname{Im}}

\newcommand{\lp}{\left(}
\newcommand{\rp}{\right)}

\DeclareMathOperator{\CM}{CM}
\DeclareMathOperator{\HCM}{HCM}

\DeclareMathOperator{\GGC}{GGC}
\DeclareMathOperator{\re}{Re}

\begin{document}

\title{HCM property and the Half-Cauchy distribution}

\author{Pierre Bosch}

\address{Laboratoire Paul Painlev\'e, Universit\'e Lille 1, Cit\'e Scientifique, F-59655 Villeneuve d'Ascq Cedex. {\em Email} : {\tt pierre.bosch@ed.univ-lille1.fr}}

\keywords{Half-Cauchy distribution - complete monotonicity - generalized gamma convolution - hyperbolically completely monotone - pick function - positive stable density}

\subjclass[2010]{60E07, 60E15}

\begin{abstract}
Let $Z_\al$ be a positive $\alpha$-stable random variable and $T_\al=(Z_\al/\tilde Z_\al)^\al,$ with independents components in the quotient. It is known that $T_\al$ is distributed as the positive branch of a Cauchy random variable with drift. We show that the density of the power transformation $T_\al^\beta$ is hyperbolically completely monotone in the sense of Thorin and Bondesson if and only if $\al\le1/2$ and $|\beta|\ge 1/(1-\al).$ This clarifies a conjecture of Bondesson (1992) on positive stable densities.
\end{abstract}

\maketitle

\section{Introduction}
A function $f:(0,\infty)\to(0,\infty)$ is said to be hyperbolically completely monotone (HCM) if for every $u>0$, the function $f(uv)f(u/v)$ is completely monotone (CM) as a function of the variable $w=v+v^{-1}$. This class coincides with that of functions of the form 
\begin{equation}
\label{prodHCM}
cx^a \prod_{i=1}^n(1 +c_ix)^{-b_i}
\end{equation}
with $a\in\R$ and $c, c_i, b_i >0,$ or pointwise limits thereof.
A positive random variable $X$ is called HCM if it has a density which is $\HCM$. The HCM class is closed with respect to multiplication and
division of independent random variables. Moreover, if $X$ is $\HCM$ then $X^\beta$ is $\HCM$ for every $|\beta|\ge1$. This class was introduced by O. Thorin and L. Bondesson and is closely connected to the class of generalized gamma convolutions (GGC). We say that the distribution of a positive random variable $X$ is a $\GGC$ if its Laplace transform reads
\begin{equation}\label{GGC}
\E{e^{-\la X}}=\exp\left(-a\la-\int_0^\infty \log\left(1+\frac\la x\right)\nu(dx)\right)
\end{equation}
for some $a\ge0$ called the drift coefficient and some positive measure $\nu$ called the Thorin measure, which is such that
\[
\int_0^1|\log(x)|\nu(dx)<\infty \et \int_1^\infty x^{-1}\nu(dx)<\infty.
\]
The GGC class is a subclass of the positive self-decomposable (SD) distributions, and in particular all GGC distributions are infinitely divisible (ID). In \cite{Article:Bondesson:HCMGGC} Bondesson proved the inclusion $\HCM\subset \GGC$, which allows to show the GGC property and hence the infinite divisibility of many positive distributions whose Laplace transforms are not explicit enough. As a genuine  example, in \cite{Article:Thorin:Bondesson} O. Thorin proves the infinite divisibility of powers of a gamma random variable at order $\xi$ with $|\xi|\ge1$.  This is also an easy consequence to the fact that gamma densities are $\HCM$. In fact Thorin uses the HCM-idea in a primitive form in his paper.

Another link between the two above classes is that a random variable $X$ is GGC if and only if its Laplace transform is $\HCM$. This characterization, which is also due to Bondesson, can be used to show both GGC and HCM properties, and it will play some role in the proof of the main result of this paper. We refer to the monograph \cite{Livre:Bondesson} for an account on these topics, including the proof of all above properties.

Let $Z_\al$ be a positive $\al$-stable random variable, $\al\in(0,1)$, normalized such that its Laplace transform reads 
\[
\E{e^{-\la Z_\al}} = \exp\left({-\la^\al}\right)=\exp\left(-\frac{\al\sin(\al\pi)}{\pi}\int_0^\infty\log\left(1+\frac\la x\right)x^{\al-1}dx\right).
\]
Observe that this Laplace tranform is of the form (\ref{GGC}), so that all positive $\al$-stable distributions are GGC. In this paper we are concerned with the following 

\begin{conjecture} [Bondesson] The density of $Z_\al$ is $\HCM$ if and only if $\al\le1/2.$ 
\end{conjecture}

This problem is stated in \cite{Report:Bondesson}, where the easy only if part is also obtained. If $\al = 1/n$ for some integer $n \ge 2,$ the HCM property for $Z_{1/n}$ is  a direct consequence of the independent factorization (see Example 5.6.2 in \cite{Livre:Bondesson})
\begin{equation}\label{factoza}
Z_{1/n}^{-1}\;\elaw\; n^n \gamma_{1/n}\times\dots\times\gamma_{(n-1)/n}
\end{equation}
where, here and throughout, $\gamma_t$ denotes a gamma random variable with shape parameter~$t$ and explicit density 
$$\frac{x^{t-1}e^{-x}}{\Ga(t)}\cdot$$ 
The if part of this conjecture is however still open when $\al\ne1/n$. In \cite{Article:Simon:MSU}, it is shown that $Z_\al$ is hyperbolically monotone (viz. its density $f$ is such that $f(uv)f(u/v)$ is non-increasing in the variable $v +1/v$) if and only if $\al\le1/2$. Proposition 4 of \cite{Article:JedidiSimon} shows that  quotient $Z_\al/\tilde Z_\al$ (with  independents components) has an HCM density if and only if $\al\le1/2.$ We refer to the whole article \cite{Article:JedidiSimon} for other partial results on Bondesson's conjecture. Last, a positive answer to the if part for $\al\in(0,1/4]\cup [1/3,1/2]$ has been recently announced in \cite{Preprint:Fourati}.

In this paper we consider the random variable 
$$T_\al \;= \;\left(\frac{Z_\al}{\tilde Z_\al}\right)^\al$$ where $Z_\al\perp \tilde Z_\al$. 
It is well-known (see \cite{Article:Zolotarev} or Exercise 4.21 in \cite{Livre:ChaumontYor}) that $T_\al$ has an explicit density which is that of a drifted Cauchy random variable conditioned to be positive, viz. it is given by 
\begin{equation}\label{DensiteTal}
\frac{\sin(\pi\al)}{\pi \al(x^2+2\cos(\pi\al)x +1)}\cdot
\end{equation}
When $\al = 1/2$, the above random variable is the half-Cauchy, whose infinite divisibility has been obtained in \cite{Article:Bondesson:IDC}. This result has been refined into self-decomposability in \cite{Article:Didier}. On the other hand, it is shown in \cite{Livre:Bondesson} that $T_\al$ has never a GGC distribution, and in particular does not have an HCM density. Our main result shows that this property holds when taking sufficiently high power transformations.

\begin{theorem}
The power transformation $T_\al^\beta$ has a $\HCM$ density if and only if $\al\le1/2$ and $|\beta|\ge 1/(1-\al)$.
\end{theorem}
Whereas the only if part of this theorem is a direct consequence of known analytical properties of HCM functions, the if part is more involved and relies on Laplace inversion and a Pick function characterization. This result shows that the explicit density of $T_\al^\beta$ is the pointwise limit of functions of the type (\ref{prodHCM}) as soon as $\al\le1/2$ and $|\beta|\ge 1/(1-\al),$ but we could not find any constructive argument for that. 

The main interest of our theorem is to propose a refined version of Bondesson's conjecture, from the point of view of power transformations. It is indeed natural to raise the further 
 
\begin{conjecture} The density of $Z_\al^\beta$ is $\HCM$ if and only if $\al\le1/2$ 
and $|\beta|\ge \al/(1-\al).$ 
\end{conjecture}
Observe that our result shows already the only if part. Some partial results for the if part are also given in \cite{Article:JedidiSimon} where it is shown that $Z_\al^\beta$ is SD when $\al\le1/2$ and $\beta\le -\al/(1-\al)$ -- see Proposition 1 in \cite{Article:JedidiSimon}, and the whole Section 3 therein where the critical power exponent $\al/(1-\al)$ appears naturally. In general, this conjecture on the power transformations of $Z_\al$ seems hard to solve, even when $\al$ is the reciprocal of an integer. In this paper we briefly handle the explicit case $\al = 1/2$ which is immediate, and the case $\al = 1/3$ which relies on a certain product formula for the modified Bessel function. This formula leads to another conjecture on the independent product of two gamma random variables. 

\section{Proof of the theorem}\label{sectionproof}

We first fix some notation and gather some known material on the CM and HCM properties. We denote by $\mathscr L\nu(x)=\int_0^\infty e^{-x\la}\nu(d\la)$ the Laplace transform of a $\sigma$-finite measure $\nu$ on $[0,\infty)$. If this measure $\nu$ is absolutely continuous with density $g,$ we write $\mathscr Lg(x)=\mathscr L \nu(x)=\int_0^\infty e^{-x\la}g(\la)d\la$. Recall that a function $f:(0,\infty)\to (0,\infty)$ is said completely monotone ($\CM$) if it is smooth and such that $(-1)^nf^{(n)}\ge 0$ for all $n\ge0.$  By Bernstein's theorem, a function $f$ is $\CM$ if and only if there exists a positive $\sigma$-finite measure $\nu$ on $[0,\infty)$ such that $f=\mathscr L\nu$ -- see e.g. chapter 1 in \cite{Livre:Schilling}. When $f(0^+)=1$, a CM function $f$ is hence the Laplace transform of some probability distribution with non-negative support.

\begin{prop} [\cite{Livre:Bondesson} p.69] \label{Prop:hcm}
Let $f:(0,\infty)\to[0,\infty)$ be an $\HCM$ function.
\begin{enumerate}
	\item $\forall a\in\R,~\forall |b|\le1, ~x\mapsto x^a f(x^b)$ is $\HCM$.\label{HCMMix}
	\item If $f(0^+)>0$ then $f$ is $\CM$.\label{HCMCM}
	\item $f$ has an analytic continuation on $\C\setminus\R_-$.\label{HCMHolo}
\end{enumerate}
\end{prop}

\begin{prop} [\cite{Livre:Bondesson} p.69] \label{Prop:hcm2}
A function $g:(0,\infty)\to(0,\infty)$ is $\HCM$ if and only if $x\mapsto g(x^a)$ is $\HCM$ for all $a \in(0,1).$
\end{prop}

\begin{prop} [\cite{Livre:Bondesson} Theorem 5.4.1]\label{Prop:HCMCM}
A probability distribution $\mu$ is a $\GGC$ if and only if its Laplace transform $\mathscr L\mu$ is an $\HCM$ function.
\end{prop}

\begin{prop}[\cite{Livre:Bondesson} Theorem 3.1.3] \label{Prop:Pick}
Let $\mu$ be a probability distribution and let $\varphi=\mathscr L\mu.$ Then $\mu$ is a $\GGC$ if and only if $\varphi$ has an analytic continuation to $\C\setminus\R_-$ such that $\varphi$ does not vanish on $\C\setminus\R_-$ and
\[
\im(z)>0\;\Rightarrow\; \im(\varphi'(z)/\varphi(z))\ge0.
\]
\end{prop}

The last proposition is known as the Pick function characterization of GGC distributions, and we refer to chapter 6 in \cite{Livre:Schilling} for more on this topic. 
We last recall for completeness the following well-known formula  -- see e.g. chapter II, Theorem 7.4 in \cite{Livre:Widder}:

\begin{prop}[Laplace inversion formula]\label{Prop:InvLaplace} Let $f:\{z\in\C/\re(z)>0\}\to\C$ be an analytic function such that:
\begin{enumerate}
		\item $f$ is real on $(0,\infty)$.
		\item $\forall c>0,~t\mapsto f(c+it)$ is integrable on $\R$.
\end{enumerate}
Then there exists an integrable function $g:[0,\infty)\to\R$ such that $f=\mathscr Lg$. Moreover $\forall c>0$, $$g(\la)=\mathscr L^{-1} f(\la)=\frac1{2\pi i}\int_{\re(z)=c} e^{\la z}f(z)dz.$$
\end{prop}

Let us consider now the function
\begin{equation}\label{falga}
f_{\al,t}(x)\; =\; \frac1{x^{2t}+2\cos(\pi\al)x^t+1}
\end{equation}
with $\al\in(0,1)$ and $t\ge0,$  and set $f_\al=f_{\al,1-\al}$. We see from a change of variable in the formula (\ref{DensiteTal}) and Proposition \ref{Prop:hcm}~(\ref{HCMMix}) that our theorem amounts to show that $f_{\al,t}$ is $\HCM$ if and only if $\al\le1/2$ and $t\le1-\al$.

\subsection{Proof of the only if part} Proposition \ref{Prop:hcm}~(\ref{HCMCM}) shows the necessity of the condition $\al\le1/2$ because otherwise the function $f_{\al, t}$ would be locally increasing at $0^+$ (this simple remark is useful in the study of the GGC property, and will be further discussed in Section \ref{Section:GGCZaZa}). To show the necessity of $t\le1-\al$ it suffices to invoke Proposition \ref{Prop:hcm}~(\ref{HCMHolo}). More precisely, let 
\begin{equation}\label{Pal}
P_\al(z)=z^2+2\cos(\pi\al)z+1.
\end{equation}
This polynomial has two zeroes $e^{\pm i(1-\al)\pi}$, so that the function $P_\al(z^t)$ vanishes on $\C\setminus\R_-$ if and only if $t> 1-\al$, the two zeroes being then $e^{\pm i(1-\al)\pi/t}$. Hence, $f_{\al,t}$ has an analytic continuation on $\C\setminus\R_-$ only if $t\le 1-\al$.

\subsection{Proof of the if part} By Proposition \ref{Prop:hcm}~ (\ref{HCMMix}), it is enough to prove that $f_{\al}$ is $\HCM$. Observe first that 
$$f_{\al}(x)\; \to\; \frac1{(x+1)^2}$$ 
as $\al\to 0,$ and that the limit is clearly an HCM density. Let us now look at two particular cases.
\subsubsection{The case $\al=1/2$} We have 
\[
f_{1/2}(x)\; =\; \frac1{x+1}
\]
which is the prototype of an $\HCM$ function. Another way to handle this case is to use the identity 
$$Z_{1/2}\elaw\frac1{4\gamma_{1/2}},$$ 
which entails that $T_{1/2}^2 \elaw \gamma_{1/2}/\gamma_{1/2}$ is $\HCM$. 

\subsubsection{The case $\al=1/3$} One has
\[
f_{1/3}(x) = \frac1{x^{4/3}+x^{2/3}+1}\cdot
\]
However, we do not know how to prove the $\HCM$ property of this function neither directly, nor by showing that $f_{1/3}$ is the pointwise limit of functions of the type (\ref{prodHCM}). On the other hand, it is clear from the above considerations that this function is HCM as soon as $\sqrt{Z_{1/3}}$ has an HCM density. Using a change of variable and formula (2.8.31) in \cite{Livre:Zolotarev}, this latter density is given by 
\[\frac2{3\pi x^{2}}\, K_{1/3}\lp\frac2{3\sqrt{3}x}\rp,\] 
where \[
K_\al(x)\;=\;\int_0^\infty \cosh(\al y)e^{-x\cosh(y)}dy, \qquad \al\in\R,
\]
is a modified Bessel function. On the other hand, the product formula (79) p. 98 in \cite{Livre:Erdelyi2} tells that for all $\al\in (-1/2,1/2),~x,y>0,$ one has
\[ K_\al(x) K_\al(y) = 2\cos(\pi\al)\int_0^\infty K_{2\al}\left(2\sqrt{xy}\operatorname{sinh}(t)\right)e^{-(x+y)\operatorname{cosh}(t)}dt.\]
Hence, for all $u,v>0,$
\[ K_\al(uv) K_\al(u/v) = 2\cos(\pi\al)\int_0^\infty\underbrace{ K_{2\al}\left(2u\operatorname{sinh}(t)\right)e^{-uw\operatorname{cosh}(t)}}_{\text{$\CM$ in $w=v +1/v$}}dt\]
and since the $\CM$ class is closed under mixing -- see chapter 1 in \cite{Livre:Schilling} -- all in all this shows that $K_\al(uv) K_\al(u/v)$ is $\CM$ in the variable $v+1/v$, which by Proposition \ref{Prop:hcm}~(\ref{HCMMix}) entails the required $\HCM$ property for $\sqrt{Z_{1/3}}.$ 

\begin{rem} {\em (a) Recall that the factorization (\ref{factoza}) reads
$$\frac1{\sqrt{Z_{1/3}}}\;\elaw \; 3\sqrt{3\ga_{1/3}\ga_{2/3}}\cdot$$
More generally, the independent product $\sqrt{\ga_t\ga_s}$ has density
\[\frac{4x^{t+s-1}}{\Ga(t)\Ga(s)}K_{t-s}(2x)\]
for all $s,t>0.$ Hence, the above product formula for the modified Bessel function shows that $\sqrt{\ga_t\ga_s}$ is $\HCM$ (hence ID) whenever $|t-s|\le1/2$. In particular, the square root of the independent product of two unit exponential random variable is infinitely divisible, a fact which seems unnoticed in the literature. It is hence  natural to raise the 

\begin{conjecture} For all $s,t>0,$ the independent product $\sqrt{\ga_t\ga_s}$ is {\em HCM}.
\end{conjecture}\noindent
Recall that $\sqrt{\ga_t}$ is not ID because of the superexponential tails of its distribution function - see Theorem 26.1 in \cite{Livre:Sato}, and hence not HCM.\\

(b) The above considerations show that our conjecture stated at the end of the introduction is true at least for $\al = 1/2$ and $\al = 1/3.$ By (\ref{factoza}), its validity for $\al = 1/n$ with $n\ge 4$ amounts to the fact that 
$$(\gamma_{1/n}\times\cdots\times \gamma_{(n-1)/n})^{1/(n-1)}$$ 
has an HCM density. More generally, we believe that the latter should be true for all independent products of the type $(\gamma_{t_1}\times\cdots\times \gamma_{t_n})^{1/n}$ with $t_1, \ldots, t_n > 0.$ The computations connected to this latter problem seem however quite challenging.\\

(c) The above remark (a) shows that the independent product of two half-Cauchy random variables
$$T_{1/2}\,\times\, T_{1/2}\; \elaw\; \frac{\sqrt{\gamma_{1/2}\gamma_{1/2}}}{\sqrt{\gamma_{1/2}\gamma_{1/2}}}$$
is HCM. This fact is less trivial than the HCM property for $T_{1/2}^2.$}
\end{rem}

\bigskip

\subsubsection{}

Let us now outline the proof of the if part. We will first show that the function $f_{\al}$ is CM. Since $f_{\al}(0)=1$, the function $f_{\al}$ is then the Laplace transform of some probability distribution $\mu_{\al}$. We then show that the latter is a $\GGC$ by applying the Pick criterion of Proposition \ref{Prop:Pick} on $f_{\al}.$ This will show that $f_{\al}$ is HCM by Proposition \ref{Prop:HCMCM}.

\begin{prop}\label{Prop:CM}The function $f_{\al}$ is $\CM$ for all $\al\le1/2.$
\end{prop}
\begin{proof} The cases $\al=0$ and $\al=1/2$ are clear from the above, and we suppose henceforth $\al\in(0,1/2)$. Since the pointwise limit of CM functions remains CM -- see e.g. chapter 1 in \cite{Livre:Schilling}, it is enough to show 
that 
$$f_{\al,1-\al-\e}(x)\;=\;\frac1{x^{2(1-\al-\e)}+ 2\cos(\pi\al)x^{1-\al-\e}+1}$$
is CM for all $\e$ sufficiently small. Fix $\e > 0$ small enough and set $f =f_{\al,1-\al-\e}$ for simplicity. For all  $c, \la > 0$ we integrate the analytic function $e^{\la z} f(z)$ along the following contour $\mathscr C.$

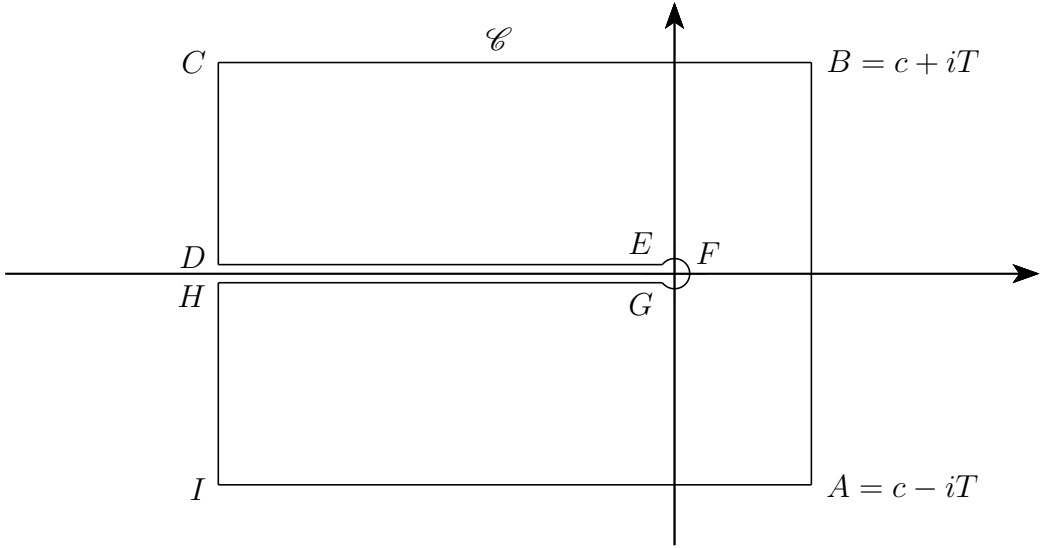
\begin{figure}[h]

\psset{unit=0.4cm 
,algebraic=true 
,linewidth=0.6pt, 
,arrowsize=6pt 2,arrowinset=0.25 
}

\begin{pspicture*}(-22,-9)(12,9) 
\psaxes[
,xAxis=true,yAxis=true
,Dx=40,Dy=20,ticksize=-12pt 12pt,subticks=1 
]
{->}(0,0)(-22,-9)(12,9)
\psline(4.5,-7)(4.5,7)
\psline(4.5,7)(-15,7)
\psline(-15,7)(-15,0.3)
\psline(-15,0.3)(-0.4,0.3)
\parametricplot{-2.4927751446322683}{2.4927751446322683}{0.5*cos(t)|0.5*sin(t)}
\psline(-0.4,-0.3)(-15,-0.3)
\psline(-15,-0.3)(-15,-7)
\psline(-15,-7)(4.5,-7)
\rput[bl](5,-7.5){$A=c-iT$}
\rput[ul](5,7){$B=c+iT$}
\rput[ur](-15.4,7){$C$}
\rput[br](-15.4,0.2){$D$}
\rput[br](-15.4,-1.1){$H$}
\rput[br](-15.4,-7.5){$I$}
\rput[ur](-0.7,1){$E$}
\rput[bl](-6.2,7.4){$\mathscr C$}
\rput[ur](-0.7,-1){$G$}
\rput[bl](0.7,0.32){$F$}
\end{pspicture*}
\caption{The contour $\mathscr C$}
\label{contour}
\end{figure}

\noindent
On the one hand, $|f(z)|\lesssim |z|^{-2(1-\al-\e)}$ as $|z|\to\infty$ and $f(z)\to1$ as $z\to0$. We deduce that
\[
\int_{[BC]\cup[CD]\cup\wideparen{EFG}\cup[HI]\cup[IA]}e^{\la z}f(z)dz\longrightarrow 0.
\]
On the other hand
\[
\frac1{2i\pi}\int_{[DE]\cup[GH]}e^{\la z}f(z)dz  \longrightarrow  \frac1\pi\int_0^{\infty}\im\left[\frac1{P_\al\left(x^{1-\al-\e} e^{i(1-\al-\e)\pi}\right)}\right]e^{-\la x}dx
\]
and
\[
\frac1{2i\pi}\int_{[AB]}e^{\la z}f(z)dz \longrightarrow \mathscr{L}^{-1}f(\la)
\]
where $\mathscr L^{-1}$ denotes the inverse Laplace transform -- see Proposition \ref{Prop:InvLaplace}. By Cauchy's theorem, this finally entails 
\[
\mathscr{L}^{-1}f(\la)=-\frac1\pi\int_0^{\infty}\im\left[\frac1{P_\al\left(x^{1-\al-\e} e^{i(1-\al-\e)\pi}\right)}\right]e^{-\la x}dx
\]
where $P_\al$ is defined by (\ref{Pal}).
To prove that $f$ is $\CM$ it suffices to show that $\mathscr{L}^{-1}f$ is a non negative function, which is equivalent to $$\forall\la\geq0,~\frac1\pi\int_0^{\infty}\im\left[\frac1{P_\al\left(x^{1-\al-\e} e^{i(1-\al-\e)\pi}\right)}\right]e^{-\la x}dx\le0.$$
Observe that
\[
\im\left(\frac1{P_\al(z)}\right) = -\frac{2\im(z)(\cos(\pi\al)+\re(z))}{|P_\al(z)|^2},
\]
so that the sign of $\im\left[1/{P_\al\left(x^{1-\al-\e} e^{i(1-\al-\e)\pi}\right)}\right]$ is negative on $(0,x_0)$ and positive on $(x_0,\infty)$ for some $x_0>0.$ The following lemma is elementary and its proof is left to the reader.

\begin{lemma}
Let $h:(0,\infty)\to\R$ an integrable function and suppose there exists $x_0>0$ such that $h$ is negative on $(0,x_0)$ and positive on $(x_0,\infty)$.
Then
\[
\int_0^{\infty}h(x)dx\le0 \;\Rightarrow \;\forall\lambda \ge0,~\int_0^\infty h(x)e^{-\la x}dx\le 0.
\]
\end{lemma}

Thus, by the lemma, it remains to show that 
$$\int_0^{\infty}\im\left[\frac1{P_\al\left(x^{1-\al-\e} e^{i(1-\al-\e)\pi}\right)}\right]dx\,\le\, 0.$$
Reasoning on the contour exactly as above, we have 
$$\int_0^{\infty}\im\left[\frac1{P_\al\left(x^{1-\al-\e} e^{i(1-\al-\e)\pi}\right)}\right]dx=-\frac1{2i\pi}\int_{\re(z)=c}f(z)dz\xrightarrow[c\to+\infty]{}0$$
Since the left part does not depend on $c>0$, this shows that it must equal $0,$ which finishes the proof of Proposition \ref{Prop:CM}.

\end{proof}

\begin{rem}\label{Remark:Kanter} {\em (a) By a well-known criterion -- see again chapter 1 in \cite{Livre:Schilling}, the function $f_{\al,t}$ is CM for all $\al \le 1/2, t \in [0, 1-\al].$ When $t\le 1/2,$ this property follows also from the immediate fact that $f_{\al,t}$ is the reciprocal of a Bernstein function, hence the Laplace transform of the potential measure of some subordinator -- see chapter 1 in \cite{Livre:Schilling} for details and terminology. On the other hand, the function $x\mapsto x^{2(1-\al)}+2\cos(\pi\al)x^{1-\al}+1$ is not Bernstein for $\al\le 1/2,$ so that the CM property of $f_\al$ cannot follow from this argument. \\

(b) One could ask if $f_\al$ is also a Stieltjes transform viz. the double Laplace transform of a positive measure, when $\al \le 1/2.$ The answer is however negative. Indeed, the Stieltjes inversion formula -- see chapter VIII Theorem 7.a in \cite{Livre:Widder} -- would entail
\[
m(dx)=-\frac1\pi\lim_{\e\to0^+} \im\left[f_\al(-x+i\e)\right]dx.
\]
and we can check that the right-hand side is not non-negative. Another way to see this is to use again the fact that $1/f_\al$ is not a Bernstein function, hence not a complete Bernstein function --~see chapter 6 in \cite{Livre:Schilling}.\\

(c) The Kanter factorization -- see Corollary 4.1 in \cite{Article:Kanter} -- reads
\[Z_\al^{-\al/(1-\al)}\; \elaw\; L \times Y_\al\]
where $L\sim\Exp(1)$ and $Y_\al$ is some independent random variable. This entails that 
$$T_\al^{1/(1-\al)}\; \elaw\; L \,\times \,Y_\al\, \times\,Z_\al^{-\al/(1-\al)}$$ 
has a density which is $\CM$, in other words that the function 
\[ x\mapsto \frac{x^{-\al}}{x^{2(1-\al)}+2\cos(\pi\al)x^{1-\al}+1}\]
is $\CM$. However, when $\al\le1/2$ this fact is weaker than Proposition \ref{Prop:CM}, which we do need in its full extent in order to apply the Pick criterion on Laplace transform of probability measures.}
\end{rem}

We can now finish the proof of the if part of the theorem. Fix $\al\in(0,1/2).$ By Proposition \ref{Prop:hcm2}, we need to show that $f(x)=f_{\al,1-\al-\e}(x)$ is HCM 
for $\e>0$ small enough. Fixing $\e >0,$ we saw during the proof of Proposition \ref{Prop:CM} that $f$ has an analytic continuation which does not vanish on $\C\setminus\R_-$. By Proposition \ref{Prop:Pick} and Proposition \ref{Prop:CM}, we hence need to check that 
\[\im(z)>0\;\Rightarrow \;\im(f'(z)/f(z))\ge 0.\]
The function $h= \im(f'/f)$ is defined on $\{z\in\C/ \im(z)>0 \}$ and is an harmonic function as the imaginary part of the analytic function $f'/f$. Besides, $h$ can be extended continuously to $\mathbb H=\{z\in\C/\im(z)\ge0\}$ and vanishes on $(0,\infty)$. Last, it is clear that $h(z)\to0$ as $|z|\to\infty$ uniformly on $\mathbb H$.
Hence, setting
$$m=\inf_{z\in\mathbb H} h(z),$$
we see that $m\in(-\infty,0]$. We will now prove that $m=0,$ which will finish the proof. On Figure \ref{FigureImff} we give two plots of the function $h$ along the lines $\{\im(z)=1\}$ and $\{\im(z)=0.1\}$ for $\al=1/5$ and $\e=1/10$.
\begin{figure}[h]
\psset{unit=1cm}
\begin{pspicture}(15,10.8)
\rput[bl](0,0){\includegraphics[width=15cm,height=11cm]{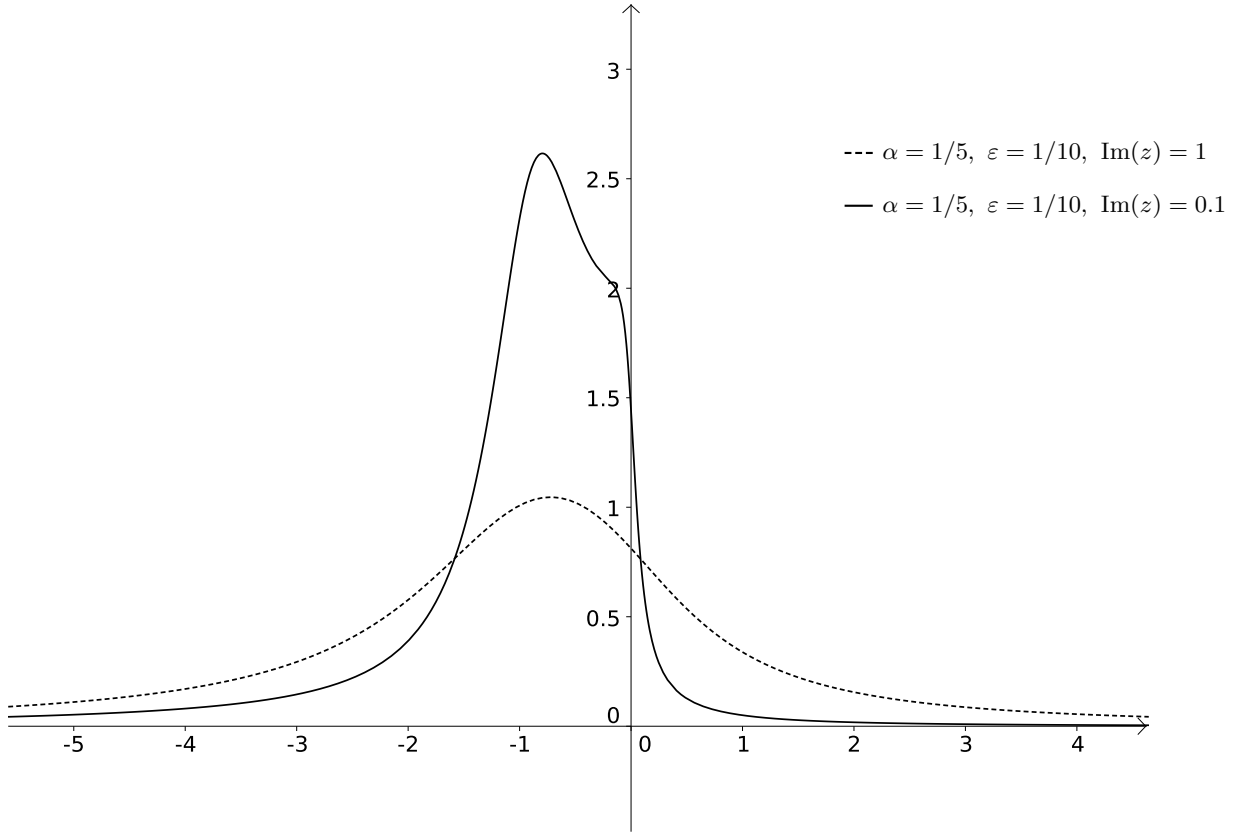}}
\rput[bl](11,9){\psline(0.09,0)\psline(0.14,0)(0.23,0)\psline(0.28,0)(0.37,0)}
\rput[bl](11,8.3){\psline(0.37,0)}
\begin{scriptsize}
\rput[l](11.5,9){$\alpha=1/5,~\e=1/10,~ \im(z)=1$}
\rput[l](11.5,8.3){$\alpha=1/5,~\e=1/10,~\im(z)=0.1$}
\end{scriptsize}
\end{pspicture}
\caption{Plot of $h$ along lines $\im(z)=c$}
\label{FigureImff}
\end{figure}

It is useless to check that $h(z)\ge0$ for all $z$ such that $\im(z)>0$. Applying the minimum principle to the harmonic function $h:\overset{\circ}{\mathbb H}\to\R$, the latter property follows as soon as $h(-x)\ge0$ for all $x>0.$ 
First, we compute, for all $z\in \mathbb H,$
\[h(z) \; = \; -2(1-\al-\e)\im\lp z^{-(\al+\e)} \frac{z^{1-\al-\e}+\cos(\pi\al)}{z^{2(1-\al-\e)}+2\cos(\pi\al)z^{1-\al-\e}+1}\rp.\]
Hence, setting $x=-\rho^{1/(1-\al-\e)}e^{i\pi}$ for some $\rho > 0$ we find
\begin{align*}
h\lp-x\rp & =- A \cdot \im\left[e^{-i(\al+\e)\pi} \lp\rho e^{i(1-\al-\e)\pi}+\cos(\pi\al) \rp\lp\rho^2e^{-i2(1-\al-\e)\pi}+2\cos(\pi\al)\rho e^{-i(1-\al-\e)\pi}+1\rp\right]\\
& =  - A \cos(\pi\al)\sin\lp(\al+\e)\pi\rp\left[ -\rho^2 + 2\frac{\cos\lp(\al+\e)\pi\rp}{\cos(\pi\al)}\rho-1 \right]\\
& = A \cos(\pi\al)\sin\lp(\al+\e)\pi\rp \left[\lp\rho-\frac{\cos\lp(\al+\e)\pi\rp}{\cos(\pi\al)}\rp^2+\underbrace{1-\lp\frac{\cos\lp(\al+\e)\pi\rp}{\cos(\pi\al)}\rp^2}_{>0}\right]
\end{align*}
with
\[A\;=\;\frac{2(1-\al-\e)\rho^{-(\al+\e)/(1-\al-\e)}}{\left| \rho^2e^{i2(1-\al-\e)\pi}+2\cos(\pi\al)\rho e^{i(1-\al-\e)\pi}+1\right|^2}\;\ge\; 0.\]
This completes the proof.

\qed

\begin{rem}\label{RHCM12} {\em (a) Writing
$$f_{\al,t}(uv)f_{\al,t}(u/v)=\frac1{u^{4t}+c^2u^{2t}+1+
c(u^{t}+u^{3t})w_{t} +u^{2t}w_{2t}}$$
with $w_{a}=v^{a}+v^{-a}$ for all $a \ge 0$ and using the fact that $w\mapsto w_{a}$ is a Berstein function when $a\in[0,1]$ -- see page 183 in \cite{Article:Bondesson:HCMGGC}, we see that the right-hand side is $\CM$ in $w$ for all $\al, t\le 1/2$. But again, this argument does not work for $t = 1-\al.$ \\

(b) The random variable defined as the independent product  
\[M_\al\; \elaw \; Z_\al\,\times\, L^{1/\al}\]
was introduced in \cite{Article:Pillai} under the denomination Mittag-Leffler random variable. In \cite{Article:AhnMcvinish} Corollary 3 and \cite{Article:JedidiSimon} Corollary 6 it is proved that $M_\al$ is HCM if and only if $\al\le1/2.$ In \cite{Article:JedidiSimon} Corollary 6 it is also shown that $M_\al$ is not hyperbolically monotone if $\al >1/2.$ As for our Theorem, it is natural to conjecture that $M_\al^\beta$ is $\HCM$ if and only if $\al\le1/2$ and $|\beta|\ge\al/(1-\al)$. \\

(c) Our result entails that the function $x\mapsto \log(x^{2t} + 2 \cos(\pi\al)x^t + 1)$ is a Thorin-Bernstein in the sense of chapter 8 in \cite{Livre:Schilling} if and only if $\al\in[0,1/2]$ and $t\in[0,1-\al].$ In other words, the function
$$x\mapsto\frac{x^{2t} +2\cos(\pi\al)x^t +1}{2x^{2t-1} + \cos(\pi\al)x^{t-1}}$$
is complete Bernstein function if and only if $\al\in[0, 1/2]$ and $t \in[0, 1-\al].$}

\end{rem}

\section{Further remarks}
\subsection{Complete monotonicity of $f_{\al,t}$} 

\label{sectioncm}

Set $\al\le 1/2.$ We know by Proposition \ref{Prop:CM} that the function $f_{\al,t}$ is $\CM$ for all $t\le1-\al$. Besides, this last constant $1-\al$ is optimal for the $\HCM$ property of $f_{\al,t}$ by our main result. Last, it is clear - see again chapter 1 in \cite{Livre:Schilling} - that there exists some $t_\al \ge 1-\al$ such that $f_{\al,t}$ is $\CM$ if and only if $t\le t_\al,$ and it is a natural question whether $t_\al =1-\al$ or not. The next proposition entails that $t_\al < 1.$

\begin{prop}\label{CM1}\label{Prop:NoCM}
The function $f_{\al,1}$ is not $\CM$ for any $\al \in (0,1).$
\end{prop}
\begin{proof}
Computing the residues of the function $z\mapsto e^{\la z}f_{\al,1}(z)$ around the rectangle~$ABCI$ of Figure \ref{contour} entails that
\[
\frac1{2i\pi}\int_{\re(z)=c} e^{\la z}f_{\al,1}(z)dz=e^{-\la a}\frac{\sin(\la b)}b
\]
with $a=\cos(\pi\al)$ and $b=\sin(\pi\al)$. Therefore, $({2i\pi})^{-1}\int_{\re(z)=c} e^{\la z}g(z)dz$ does not have a non-negative sign for all $\la>0$.
\end{proof}

The author believes that the critical index $t_\al$ should belong to the open interval $(1-\al,1),$ but he is currently unable to prove that, neither to conjecture an explicit formula for $t_\al.$ Observe that Proposition \ref{Prop:CM} shows that $T_\al^{1/(1-\al)}$ is  a gamma mixture with shape parameter $1-\al$. In other words we have the independent factorization
\[
T_\al^{1/(1-\al)}\; \elaw\; \ga_{1-\al}\times Y_\al
\]
where $Y_\al$ is some positive random variable. More generally, it is easy to see that $f_{\al,t}$ is $\CM$ if and only if $T_\al^{1/t}$ is a gamma mixture with shape parameter $t$, which means that the function
\[
s \mapsto \frac{\Ga\lp1-\frac st\rp \Ga\lp1+\frac st\rp \Ga(t)}{\Ga\lp1-\frac{\al s}t\rp \Ga\lp1+\frac {\al s}t\rp \Ga(t+s)}
\]
is the Mellin transform of some probability distribution. However, it is not easy to prove directly this latter property.
  
\subsection{$\GGC$ property for $T_\al^\beta$} 

\label{Section:GGCZaZa}

From the considerations on pp. 49-51 in \cite{Livre:Bondesson}, we observe that 
\[
T_\al^\beta\;\mbox{is a}\; \GGC~\Longrightarrow~f_{\al,1/\beta}\;\mbox{is}\;\CM
\]
for all $\beta \ge 0.$ In particular, the drifted half-Cauchy $T_\al$
is not a $\GGC$ because $f_{\al,1}$ is not $\CM$, which was already mentioned above. This also entails that $T_\al^\beta$ is not a GGC for any value of $\beta$ when $\al\in(1/2,1),$ since then $f_{\al, 1/\beta}$ is locally increasing in a neighbourhood of $0$. However, when $\al\le1/2$ it does not seem easy to characterize the $\GGC$ property for $T_\al^\beta.$ We believe that there exists some constant $\beta_\al = 1/t_\al$ such that for all $\beta > 0$ the random variable $T_\al^\beta$ is $\GGC$ if and only if $\beta\ge\beta_\al$. In general, the following conjecture from Bondesson \cite{Bondesson:2013}, which would at least entail the existence of $\beta_\al:$
\[
X\;\mbox{is a}\; \GGC\Longrightarrow X^\delta\;\mbox{is a}\; \GGC\; \forall\delta\ge1,
\]
is still open.

\bigskip
\noindent
{\bf Acknowledgements.} The author is grateful to his PhD adviser Thomas Simon for his help during the preparation of this paper. He is also grateful to Lennart Bondesson for the interest he took in this work and some useful comments.

\end{document}